\documentclass[11pt]{article}

\usepackage{latexsym,verbatim,ifthen,graphicx,mathrsfs}

\usepackage[english]{babel}
\usepackage[utf8]{inputenc} 		
\usepackage{amssymb}
\usepackage{amsmath}
\usepackage{amsthm}			
\usepackage{tikz-cd}				
\usetikzlibrary{matrix,arrows,decorations.pathmorphing}
\usepackage[all]{xy}				
\usepackage{hyperref}  			
\usepackage{anysize}								
\marginsize{2.5cm}{2.5cm}{2cm}{2cm}
\usepackage{booktabs}			
\usepackage{enumerate}			
\usepackage{multicol}

\usepackage{mathrsfs}									
\usepackage{fourier}				
\usepackage[Sonny]{fncychap}		
\usepackage{fancyhdr}

\newtheorem{theorem}{Theorem}
\newtheorem{lemma}[theorem]{Lemma}
\newtheorem{proposition}[theorem]{Proposition}
\newtheorem{corollary}[theorem]{Corollary}
\newtheorem{thm}{Theorem}

\theoremstyle{definition}	
\newtheorem{example}[theorem]{Example}
\theoremstyle{definition}
\newtheorem{remark}[theorem]{Remark}
\newtheorem{remarks}[theorem]{Remarks}

\DeclareMathOperator{\Ker}{Ker}

\DeclareMathOperator{\sgn}{sgn}

\DeclareMathOperator{\id}{id}

\newcommand{\Z}{{\mathbb{Z}}}
\newcommand{\Q}{{\mathbb{Q}}}

\newcommand{\C}{{\mathbb{C}}}
\newcommand{\K}{{\mathbb{K}}}	
\newcommand{\Lie}{{\mathbb{L}}}

\definecolor{red}{rgb}{1,0.1,0.1}

\begin{document}

\title{Formality criteria in terms of higher Whitehead brackets}

\author{Urtzi Buijs, José M. Moreno-Fernández \footnote{The authors have been partially supported by the MINECO grants MTM2013-41762-P and MTM2016-78647-P. The second author has also been supported by a Postdoctoral Fellowship of the Max Planck Society. \vskip 1pt 2010 Mathematics subject
classification: 55P62, 55Q15, 18G55.\vskip
 1pt
 Key words and phrases: Rational homotopy theory, Higher order Whitehead products, $L_\infty$-algebra, coformality, Quillen spectral sequence}} 
\date{}
\maketitle
\abstract{We provide two criteria for discarding the formality of a differential graded Lie algebra in terms of higher Whitehead brackets, which are the Lie analogue of the Massey products of a differential graded associative algebra. We also show that formality of a differential graded Lie algebra is not equivalent to the collapse of the Quillen spectral sequence. Finally, we use $L_\infty$ algebras and Quillen's formulation of rational homotopy theory to recover and improve a classical theorem for detecting higher Whitehead products in Sullivan minimal models, and give some applications.}

\section*{Introduction} Since it was first introduced in \cite[\S4]{Del75}, the notion of formality has been essential to several branches of  mathematics. Deep results where the formality of a differential graded Lie algebra (DGL, hereafter) or of a commutative differential graded algebra (CDGA, hereafter) plays a fundamental role are, for instance, Kontsevich's proof of the Quantization Theorem \cite{Kon03}, the existence of a large family of symplectic manifolds with no K\"ahler structure \cite{Del75,Cor86}, and the fact that deformation problems governed by a formal DGL have at most quadratic singularities \cite{Gol88}. 

In rational homotopy theory, which is our main concern in this paper, formality of a positively graded DGL corresponds to \emph{coformality} of the corresponding simply connected homotopy type. The concept of coformality, first introduced in \cite{Nei78}, roughly means that the rational homotopy type of any such a space $X$ is determined by its rational homotopy Lie algebra $\pi_*(\Omega X)\otimes \Q.$ Coformal spaces are essential to rational homotopy theory for several reasons. From a classical point of view, one might see coformal spaces as building blocks for rational homotopy types, since every rational simply connected homotopy type can be realized as a perturbation of the corresponding coformal model (\cite{Ouk78}). More recently, a series of works embracing \cite{Pap04,Ber14,Ber17,Ber17A} prove very interesting results, combining Koszul duality and rational homotopy theory methods, which allow for effectively computing new results on (free and based) loop space homology and other interesting topics. In these, coformality plays a distinguished role.

Our contribution with this work is giving a common perspective to three related concepts which naturally live in, or relate to, the homology of a DGL and the notion of (co)formality: the Lie analogs of the Massey products (called here higher order Whitehead brackets), the Quillen spectral sequence, and the induced $L_\infty$ structures from a homotopy retract. 

In Section \ref{Criterios}, we provide two criteria for discarding the formality of a DGL in terms of higher Whitehead brackets. In the particular case of a positively graded DGL, these criteria can be used to discard the coformality of the corresponding simply connected homotopy type. The result is folklore, yet there seems not to be written account for it. This is Thm. \ref{teorema}, see the corresponding section for the precise details.

\begin{thm}\label{TeoA} Let $L$ be a DGL and let $x_1,\dots , x_k\in H=H_*(L)$ be such that the higher Whitehead bracket set $[x_1,\dots,x_k]$ is non empty. Denote by $[-,\dots,- ]'$ the higher order Whitehead brackets of $H$, seen as a DGL with trivial differential. Then, $L$ is not formal if one of the following conditions hold:
\begin{enumerate}[{(1)}]
\item $0\notin [x_1, \dots , x_k]$.
\item The sets $[x_1, \dots , x_k]$ and $[x_1, \dots , x_k]'$ are not bijective.
\end{enumerate}
\end{thm}

We give examples of the use of Theorem \ref{TeoA} in rational homotopy theory, and characterize those finite products of odd dimensional spheres which are intrinsically coformal (Thm. \ref{IntriCo}).
\begin{thm} The product of $k$ simply connected odd dimensional spheres $S^{n_1}\times \cdots \times S^{n_k}$ is intrinsically coformal if and only if $k\leq 4$, or $k\geq 5$ and $n_i \neq n_{j_1}+\cdots + n_{j_r}-1$ for every $i$ and subset $\{n_{j_1},...,n_{j_r}\}\subseteq\left\{n_1,...,n_k\right\}$, where $r\geq 4$ is  even.
\end{thm}

That the collapse at the second page $E^2$ of the Quillen spectral sequence of a simply connected space is not equivalent to the coformality of it is accepted among rational homotopy theorists. However, an explicit example is missing in the literature. We fill this gap in Section \ref{ssQuillen} by providing Example \ref{Ejemplo}, and extend previously known results for DGL's to $L_\infty$ algebras. The following result follows from this example:

\begin{thm} There exist simply connected rational spaces all of whose higher Whitehead products vanish, their Quillen spectral sequence collapses at $E^2$, yet are not coformal.
\end{thm}

Let $(\Lambda V,d)$ be the minimal Sullivan model of a simply connected finite type complex $X$. The rational higher order Whitehead products of $X$can be read off from $d$. This is the main result of a classical paper by Andrews-Arkowitz, \cite[Thm 5.4]{Ark78}. In Section \ref{SullivanYLInf}, we derive and extend a little bit this result using $L_\infty$ structures, giving an independent proof relying on Quillen's formulation of rational homotopy theory. The notation involved in this result is a bit technical, so we refer the reader to Theorem \ref{TeoArkInf} and the subsequent Remarks \ref{Remarkillos} for its precise statement. We summarize it as the following result.

\begin{thm} Let $x_j\in H$,  $1\leq j \leq r$, be such that $\left[x_1,\dots,x_r\right]$ is defined. Let $v\in V^{N-1}$, and let $x\in \left[x_1,\dots,x_r\right]$.  If either $dv\in\Lambda^{\ge r}V$ or the corresponding homotopy retract is adapted to $x$, then $$\langle  v;sx\rangle = \varepsilon \langle  v;s\ell_r(x_1,\dots,x_r)\rangle.$$ 
\end{thm}

\noindent {\bf{Acknowledgment.}} The second author is grateful to the Max Planck Institute for Mathematics in Bonn for its hospitality and financial support.

\subsubsection*{Preliminaries} The base field for all algebraic structures is $\Q$, and gradings are taken over $\Z$. In any algebraic structure, $|x|$ denotes the degree of the element $x$.

An \emph{$L_\infty$ algebra} is a graded vector space $L=\left\{L_n\right\}_{n\in \Z}$ together with skew-symmetric linear maps $\ell_k:L^{\otimes k}\to L$ of degree $k-2$, for $k\geq 1$, satisfying the \emph{generalized Jacobi identities} for every $n\geq 1$: $$\sum_{i+j=n+1} \sum_{\sigma \in S(i,n-i) } \varepsilon(\sigma) \sgn(\sigma)(-1)^{i(j-1)} \ell_{j}\left(\ell_i\left(x_{\sigma(1)},...,x_{\sigma(i)}\right),x_{\sigma(i+1)},...,x_{\sigma(n)}\right) =0.$$ 

\noindent Here, $S(i,n-i)$ are the $(i,n-i)$ shuffles, given by those permutations $\sigma$ of $n$ elements such that $$\sigma(1)<\cdots < \sigma(i) \quad \textrm{ and } \quad \sigma(i+1)<\cdots < \sigma(n);$$ and $\varepsilon(\sigma), \sgn(\sigma)$ stand for the Koszul sign and the signature associated to $\sigma$, respectively. A \emph{differential graded Lie algebra}, DGL henceforth, is an $L_\infty$ algebra $L$ for which $\ell_k=0$ for all $k\geq 3$. In this case $\ell_1=\partial$ and $\ell_2 = [\ ,\ ]$ are the \emph{differential} and the \emph{Lie bracket}, respectively. 

An $L_\infty$ algebra is \emph{minimal} if $\ell_1=0,$ and \emph{reduced} if $L_n=0$ for every $n\leq 0$. The \emph{homology} of an $L_\infty$ algebra is the homology of the underlying chain complex $\left(L,\ell_1\right)$. An \emph{$L_\infty$ morphism} $f:L \to L'$ is a family of skew-symmetric linear maps $\left\{ f_n: L^{\otimes n} \to L'\right\}$ of degree $n-1$ such that the following equation is satisfied for every $n\geq 1$:
\begin{equation}\begin{split} \sum_{i+j=n+1} \ \ \sum_{\sigma \in S(i,n-i)} \varepsilon(\sigma)\sgn(\sigma)(-1)^{i(j-1)} f_j\left(\ell_i\left(x_{\sigma(1)},...,x_{\sigma(i)}\right),x_{\sigma(i+1)},...,x_{\sigma(n)}\right) = \\ 
\sum_{\substack{k\geq 1 \\ i_1+\cdots +i_k =n \\ \tau \in S(i_1,...,i_k)}} \varepsilon(\sigma)\sgn(\sigma)\varepsilon_k \ell_k'\left(f_{i_1}\otimes \cdots \otimes f_{i_k}\right) \left(x_{\tau(1)}\otimes \cdots \otimes x_{\tau(n)}\right),
\end{split}\label{EcuLinf}\end{equation} with $\varepsilon_k$ being the parity of $\sum_{l=1}^{k-1}(k-l)(i_l-1)$. We will reinterpret $L_\infty$ morphisms more compactly in Thm. \ref{linfinito}. An $L_\infty$ morphism is said to be an \emph{$L_\infty$ quasi-isomorphism} if $f_1:\left(L,\ell_1\right)\to\left(L',\ell_1'\right)$ is a quasi-isomorphism of chain complexes, and an \emph{isomorphism} if there exist an inverse $L_\infty$ morphism for $f$. A quasi-isomorphism of minimal $L_\infty$ algebras is an isomorphism. \\

Let $V$ be a graded vector space. In this paper, with the exception of Section \ref{SullivanYLInf}, $\Lambda V$ denotes the cocommutative coalgebra whose diagonal is the unique morphism of graded algebras turning $V$ into its primitives. In Section \ref{SullivanYLInf}, the notation $\Lambda V$ stands for the free commutative graded algebra on $V$. There will be no risk of confusion with the notation, as it will be always clear from the context if we are treating algebras or coalgebras. We write CDGC or CDGA for cocommutative differential graded coalgebra or algebra, respectively. The \emph{suspension} of $V$ is the graded vector space $sV$ satisfying $(sV)_i=V_{i-1}.$ We take for granted the notation $\Lambda^+V, \Lambda^{\leq k}V, \Lambda^k V$, see \cite[Chap. 3]{Yve12}. 

\begin{theorem}\label{linfinito}{\em \cite{Lad95}} (1) $L_\infty$ structures on the graded vector space  $L$ are in bijective correspondence with codifferentials on the coalgebra $\Lambda  sL$.

(2) $L_\infty$ morphisms from $L$ to $L'$ are in bijective correspondence with CDGC morphisms $(\Lambda sL,\delta)\to (\Lambda sL',\delta')$, where $\delta$ and $\delta'$ are the codifferentials equivalent to the corresponding $L_\infty$ structures.
 \end{theorem}
 
 \begin{proof} We give a sketch of proof for future reference.
 
 (1) A codifferential $\delta$ on $\Lambda sL$ is determined by a degree
$-1$ linear map $ \Lambda^+ sL\to sL$, written as the sum of
linear maps $h_k\colon\Lambda^ksL\to sL$, $k\ge 1$. In fact, $\delta$ is written as the sum of coderivations,
\begin{equation}\label{Coder1}
\delta=\sum_{k\ge 1}\delta_k,\quad \delta_k\colon \Lambda sL\to \Lambda sL,
\end{equation}
each of which is the extension as a coderivation of the corresponding $h_k$,
\begin{equation}\label{Coder2}
\delta_k\left(sx_1\wedge...\wedge sx_p\right)=\sum_{i_1<\dots<i_k}\varepsilon\, h_k\left(sx_{i_1}\wedge...\wedge sx_{i_k}\right)\wedge sx_1\wedge...\widehat{ sx}_{i_1}...\widehat{sx}_{i_k}...\wedge sx_p.
\end{equation}
Each $\delta_k$ decreases word length by $k-1$. The operators $\left\{\ell_k\right\}_{k\ge 1}$ on $L$ and the maps $\left\{h_k\right\}_{k\ge 1}$ (and hence $\delta$) uniquely determine  each other by:
\begin{equation}\label{CorresL}
\begin{split}
\ell_k&=s^{-1}\circ h_k\circ s^{\otimes k}\colon  L^{\otimes^k}\to
L,\\[0.2cm]
h_k&=(-1)^{\frac{k(k-1)}{2}}s\circ\ell_k\circ\left(s^{-1}\right)^{\otimes k}\colon\Lambda^ksL\to sL.
\end{split}
\end{equation}

(2) A CDGC morphism
$$f\colon (\Lambda sL,\delta)\longrightarrow(\Lambda sL',\delta')
$$
is determined by  $\pi
f\colon \Lambda sL\to sL'$ ($\pi$ denotes the projection onto the indecomposables) which can be written as
$ \sum_{k\ge 1}(\pi f)^{(k)}$, where $(\pi f)^{(k)}\colon \Lambda^ksL\to sL'$. Note that the collection of linear maps $\left\{(\pi f)^{(k)}\right\}_{k\ge 1}$ is in one-to-one correspondence with a system $\left\{f_{k}\right\}_{k\ge1}$ of
skew-symmetric maps $f_k\colon  L^{\otimes^k}\to L'$ of degree $1-k$ satisfying equations (\ref{EcuLinf}). Indeed, each $f_k$ and $(\pi f)^{(k)}$ determines the other by:
$$
\begin{aligned}
f_k&=s^{-1}\circ (\pi f)^{(k)}\circ s^{\otimes k},\\
(\pi f)^{(k)}&=(-1)^{\frac{k(k-1)}{2}}s\circ f_k\circ \left(s^{-1}\right)^{\otimes k}.\\
\end{aligned}
$$

\end{proof}

\begin{remark}\label{QuillenChains} \rm Whenever $L$ is a DGL with differential $\partial$, the corresponding CDGC $(\Lambda sL,\delta)$ given by (1) of Theorem \ref{linfinito} is precisely the classical construction of \cite{Qui69}, see also \cite[Chap. 22]{Yve12}. That is, $\delta=\delta_1+\delta_2$ where these are produced as before by,
$$
h_1(sx)=-s\partial x,\quad h_2(sx\wedge sy)=-(-1)^{|x|}s[x,y].
$$
 Thus,
from now on, given a general $L_\infty$ algebra $L$, we denote by $\mathscr C(L)=(\Lambda sL,\delta)$ the corresponding CDGC  and call it the {\em (Quillen) chains on $L$}.
\end{remark}

A DGL $L$ is \emph{formal} if there exists a zig-zag of DGL quasi-isomorphisms
\begin{center}
\begin{tikzcd}
L   & \cdots \ar[swap]{l}{\simeq} \ar{r}{\simeq} &  H,
\end{tikzcd} 
\end{center} where $H=H_*(L)$ is endowed with the trivial differential. The reader is referred to the monograph \cite{Yve12} for the notions of rational homotopy theory involved in this paper. A simply connected space $X$ is \emph{coformal} if the DGL $\lambda(X)$ given by Quillen's classical construction \cite{Qui69} is formal. In this case, $\pi_*\left(\Omega X\right)\otimes \Q$ endowed with the Samelson bracket is a DGL model of $X$.

\begin{proposition}  A DGL $L$ is formal if and only if it is $L_\infty$ quasi-isomorphic to a minimal $L_\infty$ algebra $L'$ for which $\ell_n=0$ for all $n\geq 3.$
\end{proposition}

Note that necessarily $L'=H(L)$.

\begin{proof}It is well  known, see for instance \cite[Theorem~4.6]{Kon03} or \cite[Section~10.4]{Lod12}, that if $f:M \to N$ is an $L_\infty$ quasi-isomorphism, then there exists another $L_\infty$ quasi-isomorphism $g:N\to M$ such that $$H(f_1)\colon H(M,\ell_1)\rightleftarrows H(N,\ell_1)\colon H(g_1)$$ are inverses of each other. Hence, being quasi-isomorphic is a well defined concept for a pair of $L_\infty$ algebras. In particular, if both $M$ and $N$ are DGL's, a zig-zag
\begin{center}
\begin{tikzcd}
M  & \cdots \ar[swap]{l}{\simeq} \ar{r}{\simeq} &  N
\end{tikzcd} 
\end{center}
of DGL quasi-isomorphism exists if and only if $M$ and $N$ are quasi-isomorphic as $L_\infty$ algebras, that is, if there are $L_\infty$ quasi-isomorphisms $f\colon M\rightleftarrows N\colon g$ as before. 

If $L$ is formal, then the above discussion implies that there are $L_\infty$ quasi-isomorphisms
\begin{center}
\begin{tikzcd}
F:L  \ar[shift left]{r}{} & H \ar[shift left]{l}{}:G.
\end{tikzcd}
\end{center}
The converse follows from the fact that an $L_\infty$ quasi-isomorphism between two DGL's is equivalent to a zig-zag of DGL quasi-isomorphisms between them \cite[Thm. 11.4.14]{Lod12}.
\end{proof}

We finish the Introduction by extending to any DGL a well known fact on rational homotopy theory for reduced DGL's with finite type homology. Let $L$ be a reduced finite type DGL which is $L_\infty$ quasi-isomorphic to a minimal $L_\infty$ algebra with vanishing $\ell_k$, for $k\ge 3$. By Theorem \ref{linfinito}, there is a CDGC quasi-isomorphism
$$
(\Lambda sL,\delta) \stackrel{\simeq}{\longrightarrow}(\Lambda sH,\delta).
$$
Dualizing, we obtain a CDGA quasi-isomorphism,
$$
(\Lambda (sH)^\sharp,d)
\stackrel{\simeq}{\longrightarrow} {\mathscr C}^*(L)
$$
where, on the right, we have the classical cochains on $L$, dual of its chains (see Remark \ref{QuillenChains}) and on the left, we use the  isomorphism $(\Lambda sH,\delta)^\sharp\cong(\Lambda (sH)^\sharp,d)$ \cite[Lemma 23.1]{Yve12}. For each $k\ge 2$, the bracket $\ell_k$ on $H$ (which defines $\delta$ and vanishes for $k\ge 3$) is identified with the $k$th part $d_k$ of the differential $d$  by the formula (\ref{pairing}), see Section \ref{SullivanLInfinity}. In particular, $d$ is decomposable and just quadratic, that is, $(\Lambda (sH)^\sharp,d)$ is a Sullivan minimal model with quadratic differential (see \cite[\S 12]{Yve12}). Thus, Theorem \ref{linfinito} implies the following well known fact (\cite[II.7(6)]{Tan83}), which will be used in Section \ref{ssQuillen}.

\begin{corollary}\label{SulCuad} A finite type reduced DGL $L$ is formal if and only if the Sullivan minimal model on the  cochains  of $L$ admits a quadratic differential. In particular, if $X$ is a simply connected finite type CW-complex and $L$ is a finite type DGL model of $X$, then $X$ is coformal if and only if there is a purely quadratic differential on its Sullivan minimal model. 
\end{corollary}

It follows that, in order to detect non-formality of a DGL $L$ (respectively, non-coformality of a space $X$), it is \emph{not} enough to find a non-vanishing bracket $\ell_n$ of order greater or equal than $3$ in a particular minimal $L_\infty$ algebra quasi-isomorphic to $L$  (respectively, to a Lie model of $X$).

\section{Formality and higher Whitehead brackets}\label{Criterios}

In this section, we prove Theorem \ref{TeoA}, providing two criteria for discarding the formality of a DGL $L$. In particular, whenever $L$ is reduced, these criteria allow for discarding the coformality of simply connected spaces. These criteria are given in terms of the Lie analogs of the Massey products of a differential graded algebra (\cite{Mas58}). We call these algebraic higher products \emph{higher (order) Whitehead brackets}, see Remark \ref{Remars}. These are also known as \emph{Massey brackets}, or \emph{Massey-Lie brackets}. The higher Whitehead brackets live in the homology of DGL's. We will introduce these higher products by means of a very special DGL. Since no connectivity restrictions are imposed on the DGL's, we need to explain some facts of homotopical algebra. After the proof of Theorem \ref{TeoA}, which is Theorem \ref{teorema} in this section, we give examples of its application to rational homotopy theory.\\

For $k\geq 2$, fix integers $n_1,...,n_k$. Let $\Lie(U)$ be the free graded Lie algebra with 
\begin{equation*}
U = \left\langle u_{i_1...i_s} \mid 1 \leq i_1 < \cdots < i_k  \leq k, \ 1 \leq s < k\right\rangle,
\end{equation*} of degrees  $|u_{i_1...i_s}|=n_{i_1}+\cdots +n_{i_s}-1$ and differential $\partial$ for which $u_1,...,u_k$ are cycles, and for $s \geq 2$,
\begin{equation}\label{Diferencial}
\partial u_{i_1...i_s} = \sum_{p=1}^{s-1} \sum_{\sigma\in \widetilde S(p,n-p)} \varepsilon(\sigma) \left[u_{i_{\sigma(1)}...i_{\sigma(p)}},u_{i_{\sigma(p+1)}...i_{\sigma(s)}}\right].
\end{equation} Here, $\widetilde S(p,n-p)$ are the $(p,n-p)$ shuffles fixing $1$, and $\varepsilon(\sigma)$ is given by the Koszul sign times the parity of $|u_{i_{\sigma(1)}...i_{\sigma(s)}}|.$ Let $L$ be a DGL, and let $x_i\in H_{n_i}(L)$ for $1\leq i \leq k$. Define $\varphi:\Lie\left(u_1,...,u_k\right) \to L$ by mapping $u_i$ to a representative $\varphi(u_i)$ of $x_i$. Denote $N=n_1+\cdots + n_k$. The \emph{$k$th order Whitehead bracket set} is the (possibly empty) set $$\left[x_1,...,x_k\right]=\left\{\overline{\phi(w)} \mid \phi:\Lie\left(U\right)\to L \ \textrm{ is a DGL extension of } \varphi\right\}\subseteq H_{N+1}(L),$$ as depicted by the diagram:
\begin{equation} \label{ProdLie}
\begin{tikzcd}
 \Lie\left(u_1,...,u_k\right) \ar{r}{\varphi} \ar[hook]{d}& L \\
 \Lie\left(U\right) \ar[dotted, swap]{ur}{\phi}
\end{tikzcd}  
\end{equation}

\begin{remark}\label{Remars}\rm Whenever $L$ is reduced, $\Lie(U)$ above is the Lie model of a fat wedge of spheres, and there is a bijection between the rational topological higher order Whitehead products of a simply connected space $X$ and the higher order Whitehead brackets of a Quillen model of it (\cite{Tan83}).
\end{remark}

Recall, see for instance \cite[\S2]{Hin97}, that the category of DGL's admits a model category structure in which fibrations are the surjective morphisms, weak equivalences are the quasi-isomorphisms and cofibrations are the morphisms satisfying the left lifting property with respect to trivial fibrations.
The proof of the well known {\em Lifting Lemma} \cite[Theorem II.5. (13)]{Tan83} to surjective quasi-isomorphisms works in this case to prove:

\begin{lemma}\label{cofib}  The DGL $\Lie(U)$ is cofibrant. That is, given a trivial fibration, i.e., a surjective quasi-isomorphism
 $ \varphi:L\stackrel{\simeq}{\twoheadrightarrow} L'$, any morphism $\eta \colon \Lie(U)\to L'$ lifts to  a DGL morphism  $\theta \colon \Lie(U)\to L$ such that $\eta = \varphi \theta$.
\end{lemma}

In particular, a well known general fact on model categories shows the following.

\begin{lemma}\label{Biy} Given a (non necessarily surjective) DGL quasi-isomorphism $\varphi\colon L\xrightarrow{\simeq} L'$,  composition with  $\varphi$ induces a bijection between the sets of homotopy classes,
$$
\varphi_\sharp \colon \bigl[ \Lie(U),L \bigr]\stackrel{\cong}{\to}\bigl[ \Lie(U),L' \bigr].
$$
\end{lemma}

In this model structure, a {\em path object} for a given DGL $L$ is the DGL
$
L\otimes \Lambda (t,dt)$ in which $|t|=0$, together with morphisms
$$
\xymatrix{ L\otimes \Lambda (t,dt) \ar@<0.75ex>[r]^-{\varepsilon_0} &L, \ar@<0.75ex>@{<-}[l]^(0.32){\varepsilon_1} }
$$
defined as the identity on $L$ and $\varepsilon_0(t)=0$, $\varepsilon_1(t)=1$.
Then, two morphisms $f,g\colon \Lie(U)\to L$ are homotopic if there is a morphism $\Psi\colon \Lie(U)\to L\otimes \Lambda (t,dt)$ such that $\varepsilon_0\Psi=f$ and $\varepsilon_1\Psi=g$. In particular, as $\Lambda (t,dt)$ is acyclic,  two homotopic morphisms induce the same morphism on homology. Denote by $f_*=H_*(f)$ the map induced in homology by a DGL morphism $f$.

\begin{proposition}\label{invarianza}  Let $\varphi : L\xrightarrow{\simeq }L'$ be a  DGL quasi-isomorphism and let  $x_1,\dots ,x_k\in H_*(L)$, $k\ge 2$. Then:
\begin{enumerate}[{(1)}]
\item $0\in [x_1, \dots , x_k]$  if and only if  \ $0\in \left[\varphi_*(x_1),\dots ,\varphi_* (x_k)\right]$.
\item The map $\varphi_*$ induces a bijection between the higher order Whitehead brackets sets. In particular,
\begin{equation*}
\# [x_1, \dots , x_k]=\#\left[\varphi_*(x_1),\dots ,\varphi_* (x_k)\right].
\end{equation*}
\end{enumerate}
\end{proposition}
\begin{proof} First, observe that $[x_1, \dots , x_k]$ is non empty if and only if $[\varphi_*(x_1),\dots ,\varphi_* (x_k)]$ is also non empty. The result is now a straightforward consequence of lemmas \ref{cofib} and \ref{Biy}. Indeed, given $\overline{\phi(w)}\in [x_1,...,x_k]$, the fact that $\varphi_*$ is an isomorphism implies that $\overline{\varphi\circ \phi (w)}\in \left[\varphi_*(x_1),\dots ,\varphi_* (x_k)\right]$ for a unique homology class. And conversely, given $\overline{\phi'(w)}\in \left[\varphi_*(x_1),\dots ,\varphi_* (x_k)\right]$, there exists a unique (up to homotopy) DGL morphism $\theta:\Lie(U)\to L$ such that $\varphi\circ\theta \simeq \phi'$, which in particular implies that $\overline{\theta(w)}\in[x_1,...,x_k]$ is unique with the property that $\overline{\varphi\circ\theta(w)}=\overline{\phi'(w)}.$
\end{proof}

\begin{lemma}\label{lema1}  Let $L$ be a DGL with zero differential and let $x_1, \dots , x_k\in L$. If  $[ x_1,\dots , x_k]\not= \emptyset$, then $0\in [ x_1,\dots , x_k]$.
\end{lemma}

\begin{proof}Consider the diagram (6), where $\varphi(u_i)$ is a representative of each $x_i,$ and define $$\phi\left(u_{i_1...i_s}\right)=0 \quad \forall \ 1 \leq i_1 < \cdots < i_s \leq k, \quad 2\leq s\leq k-1.$$ Then, $\partial \phi \left(u_{i_1...i_s}\right)= \partial 0 = 0$, and
\begin{align*}
\phi\partial  \left(u_{i_1...i_s}\right) &= \phi \left(\sum_{p=1}^{s-1}\sum_{\widetilde S(p,s-p)} \varepsilon(\sigma) \left[ u_{i_{\sigma(1)}...i_{\sigma(p)}}, u_{i_{\sigma(p+1)}...i_{\sigma(s)}} \right]\right)\\
 &= \sum_{p=1}^{s-1}\sum_{\widetilde S(p,s-p)} \varepsilon(\sigma) \left[ \phi\left(u_{i_{\sigma(1)}...i_{\sigma(p)}}\right), \phi\left(u_{i_{\sigma(p+1)}...i_{\sigma(s)}}\right) \right]=0.
\end{align*} Since $\phi(w)=0,$ the result follows.
\end{proof}

The following is an immediate consequence of Proposition \ref{invarianza} and Lemma \ref{lema1}:

\begin{theorem}\label{teorema} Let $L$ be a DGL and let $x_1,\dots , x_k\in H=H_*(L)$ be such that $[x_1,\dots,x_k]$ is non empty. Denote by $[-,\dots,- ]'$ the higher order Whitehead brackets of $H$. Then, $L$ is not formal if one of the following conditions hold:
\begin{enumerate}[{(1)}]
\item $0\notin [x_1, \dots , x_k]$.
\item The sets $[x_1, \dots , x_k]$ and $[x_1, \dots , x_k]'$ are not bijective.
\end{enumerate}
\end{theorem}

Two easy to check instances where the zero element criterion works are the complex projective spaces $\C P^n$ for $n\geq 2$, and fat wedges $T\left( S^{n_1},..., S^{n_k}\right)$ of $k\geq 3$ simply connected spheres. In some cases, the zero element criterion cannot be applied, is not conclusive, or maybe one only has information about some particular cycles at hand. Then, one may use the cardinality criterion. The following example illustrates this point.

\begin{example}\rm Let $X=\left(\left(S^3\times S^3\times S^3\times S^3\right) \vee S^6\right) \cup \left\{ e_a,e_b,e_c \right\}$, where the $9$-cells $e_a,e_b$ and $e_c$ have been attached by the following Whitehead products, respectively: $$\left[\id_{ S^6},\id_{ S^3_2}\right], \quad \left[\id_{ S^6},\id_{ S^3_3}\right] \quad \textrm{ and } \quad \left[\id_{ S^6},\id_{ S^3_4}\right].$$ Here, the subindex $i$ of $\id_{ S^3_i}$ indicates that this map is the composition of a generator of $\pi_3\left( S^3\right)\cong \Z$ with the inclusion as the $i$th copy inside the product $ S^3\times S^3\times S^3\times S^3.$ The space $X$ is not coformal.
\end{example}
\begin{proof}  In view of the cellular decomposition given, the Quillen minimal model of $X$ is  $$\Lie\left(v_1,v_2,v_3,v_4,v_{12},v_{13},v_{14},v_{23},v_{24},v_{34},v_{123},v_{124},v_{134},v_{234},v_{1234},z,a,b,c\right),$$ where
$$ |v_{i_1...i_s}|	=3s - 1 \quad \textrm{ for every } \quad 1 \leq i_1 < \cdots < i_s \leq 4 \quad \textrm{ and } \quad 1 \leq s \leq 3,  \quad \textrm{ and }$$ 
$$ |z|	= 5, \qquad |a|=|b|=|c| = 8.$$ The differential is given on the elements $v_{i_1...i_s}$ by equation (\ref{Diferencial}) for every $1 \leq i_1 < \cdots < i_s \leq 4$, and $$\partial z = 0, \quad \partial a = [z,v_2], \quad \partial b= [z,v_3], \quad \textrm{ and } \quad \partial c	= [z,v_4].$$ We claim that $\left[ v_1,v_2,v_3,v_4\right]=\left\{0\right\}.$ Indeed, any extension $\phi$ in the corresponding diagram is of the form
\begin{align*}
\phi(u_i)			&=v_i 						& \textrm{for } i = 1,2,3,4,\\
\phi(u_{ij})		&= v_{ij} 					& \textrm{for } 1 \leq i < i \leq 4,\\
\phi(u_{ijk})		&= v_{ijk} + c_{ijk}			& \textrm{for } 1 \leq i < j < k \leq 4.
\end{align*} Here, $c_{ijk}$ can be any degree 8 cycle. But since $H_8\left(L\right)=0$, each $c_{ijk}$ is necessarily a boundary $\partial b_{ijk}$, of which there are many. The four-fold Whitehead product $\left[ v_1,v_2,v_3,v_4\right]$ is then given by all those homology classes represented by 
\begin{align*}
\phi(w)	&= [v_{123},v_4]-[v_{124},v_3]+[v_{12},v_{34}]+[v_{14},v_{23}]+[v_1,v_{234}]-[v_{13},v_{24}]+[v_{134},v_2]\\
		& \ \ + \partial\left([b_{123},v_4]-[b_{124},v_3]+[v_1,b_{234}]+[b_{134},v_2]\right)\\
		&= \partial v_{1234} + \partial c.
\end{align*} Since the cycle above is a boundary, the claim follows.  This shows that the zero element criterion cannot be applied for discarding the coformality of $X$, at least for the election of cycles done. However, the chosen cycles indeed discard the coformality of $X$. We show next that the corresponding Whitehead bracket set in homology has many non trivial homology classes. In fact,  we prove that $$\left[\bar v_1,\bar v_2,\bar v_3,\bar v_4\right]'=\left\{ \alpha\left[\bar z, \bar z\right] \mid \alpha \in \Q \right\},$$ and then an application of Thm. \ref{teorema} (2) finishes the proof. Note that the possible choices for the extensions (now onto the homology $H=H(L)$) are given by:
\begin{align*}
\phi(u_{ij})				&= \lambda_{ij} \bar z \quad \textrm{ for any } \lambda_{ij}\in \Q, \quad 1 \leq i < j \leq 4,\\
\phi(u_{ijk})				&=  0 \quad \ \ \quad \textrm{ for every } \quad \quad \ 1 \leq i < j < k \leq 4.
\end{align*}
\noindent This is so, since there is a unique non trivial homology class in $H_5$, represented by $z$, and the differential commutes for this choice:
$$ \partial \phi (u_{ij}) = 0 = \overline{\partial v_{ij}} = \overline{[v_i,v_j]}= \phi\left([u_i,u_j]\right) =\phi\partial (u_{ij}).$$
On the other hand, there are no non trivial cycles in $L_8$, and therefore $ 0$ is the only possible choice for $\phi(u_{ijk}).$ Therefore, $\left[\overline v_1,\overline v_2,\overline v_3,\overline v_4\right]'$ consist of all those homology classes represented by 
$$ \phi(w) = \left[\lambda_{12}\bar z, \lambda_{34}\bar z\right] + \left[\lambda_{14}\bar z, \lambda_{23}\bar z\right] + \left[\lambda_{13}\bar z, \lambda_{24}\bar z\right] = \big(\lambda_{12}\lambda_{34} + \lambda_{14}\lambda_{23} + \lambda_{13}\lambda_{24} \big) \left[\bar z, \bar z\right].$$ Since there are many non trivial classes above, it follows that $\# \left[\bar v_1,\bar v_2, \bar v_3, \bar v_4\right]'\  \neq \ \# \left[\bar v_1,\bar v_2, \bar v_3, \bar v_4\right],$ proving that $X$ is not coformal.
\end{proof}

\subsection{Intrinsic coformality}

We characterize the intrinsic coformality of a finite product of simply connected odd dimensional spheres by combining $L_\infty$ structures with higher Whitehead products. 

Recall that a connected space $X$ is \emph{intrinsically formal} if any connected space whose rational cohomology  algebra is isomorphic to $H^*(X;\Q)$ has the same rational homotopy type as $X$. In other words, if there is a unique rational homotopy type whose rational cohomology algebra is isomorphic to $H^*(X;\Q).$ Dually, a simply connected space $X$ is \emph{intrinsically coformal} if any simply connected space whose rational homotopy Lie algebra is isomorphic to $\pi_*(\Omega X)\otimes \Q$ has the same rational homotopy type as $X$. In other words, if there is a unique rational homotopy type whose rational homotopy Lie algebra is isomorphic to $\pi_*(\Omega X)\otimes \Q.$ It is well known (see \cite{Nei78}, or \cite[Lemma 1.6]{Hal79}) that the wedge $S^{n_1}\vee \cdots \vee S^{n_k}$ of $k$ simply connected spheres is formal, and it is a theorem of Baues (\cite{Bau77}, see also \cite[Thm. 1.5]{Hal79}) that it is intrinsically formal if every $n_i$ is odd. Dually, (\cite{Nei78}), the product $S^{n_1}\times \cdots \times S^{n_k}$ is coformal for any $k$ simply connected spheres. We prove the following refinement.

\begin{theorem}\label{IntriCo} The product of $k$ simply connected odd dimensional spheres $S^{n_1}\times \cdots \times S^{n_k}$ is intrinsically coformal if and only if 
\begin{enumerate}[{(1)}]
	\item $k\leq 4$, or
	\item $k\geq 5$ and	\begin{center}$n_i \neq n_{j_1}+\cdots + n_{j_r}-1$ for every $i$ and subset $\{n_{j_1},...,n_{j_r}\}\subseteq\left\{n_1,...,n_k\right\}$, where $r\geq 4$ is  even.\end{center}
\end{enumerate} 
\end{theorem}

\begin{proof} Denote $P=S^{n_1}\times \cdots \times S^{n_k}$. Let $L$ be an $L_\infty$ algebra whose underlying graded vector space is generated by homogeneous $x_1,...,x_k$ of degree $|x_i|=n_i -1$, and that $\ell_1=\ell_2=0.$ Hence, $L\cong \pi_*(\Omega P)\otimes \Q$ as graded Lie algebras. Every odd dimensional bracket $\ell_{2n+1}$ vanishes, since $$|\ell_{2n+1}\left(x_{i_1},...,x_{i_{2n+1}}\right)|=\underbrace{|x_{i_1}|+\cdots+|x_{i_{2n+1}}|}_{\textrm{even}} + \underbrace{2n+1-2}_{odd}\in L_{\textrm{odd}}=0.$$

For $k\leq2$ spheres, the result is straightforward. For $k=3$ (resp. $k=4$), the fact that even dimensional brackets $\ell_{2n}$ vanish whenever two arguments are linearly dependent, and that $\dim L = 3 $ (resp. $\dim L = 4$) implies that the whole $L_\infty$ structure is trivial. The homotopy type represented by $L$ is precisely $P$, which is therefore intrinsically coformal. 

Let $k\geq 5$. If $n_i = n_{j_1}+\cdots + n_{j_r}-1$,  endow $L$ with the $L_\infty$ structure all of whose brackets vanish except for $$\ell_{r}\left(x_{j_1},...,x_{j_r}\right)=x_i.$$ Observe that Proposition 3.1 of \cite{HigherWhitehead} readily implies that, for a given $L_\infty$ structure $\{\ell_n\}$ on the homology of a DGL with $\ell_i$ vanishing up to $i=k-1$, with $k\geq2$, such that the Whitehead bracket set $[x_1,...,x_k]$ is non-empty, it follows that $[x_1,...,x_k]$ consists of the single homology class given by $\ell_k(x_1,...,x_k)$. Hence, the homotopy type represented by $L$ carries the non trivial $r$th order Whitehead product $[x_{j_1},...,x_{j_r}]=\{x_i\}$, and it is therefore not coformal. This prevents $P$ from being intrinsically coformal.

If on the other hand, $n_i \neq n_{j_1}+\cdots + n_{j_r}-1$ for any choice of $i$, $r$ even and $j_k$, then every bracket $\ell_n$ vanishes. Indeed, assume that for some even integer $r$ (necessarily $4\leq r \leq k$) and for some $z_i \in L$ we have that $\ell_r\left(z_1,...,z_r\right)=z_{r+1}\neq 0.$ Then, the elements $z_1,...,z_r$ are linearly independent, so after a change of basis we may assume that  $\{z_1,...,z_r\} = \{x_{j_1},...,x_{j_r}\} \subseteq \{x_1,...,x_k\}$. But then, for some index $i$, $$|\ell_r\left(z_1,...,z_r\right)|=(n_{j_1}-1)+\cdots + (n_{j_r}-1)+r-2 = n_i-1,$$ a contradiction.
\end{proof}

\begin{remark}\rm When $P$ is not intrinsically coformal, we have proven the existence of a \emph{finite} complex with the same homotopy Lie algebra. This is because $L$ is concentrated in even degrees, so $\mathcal C(L)=\Lambda sL$ has graded vector space of generators $sL$ of odd degree, forcing $\mathcal C(L)$ to be finite dimensional. So, $\mathcal L \mathcal C (L)=\Lie(s^{-1}\Lambda^+sL)$ is of finite type. The classical approach (i.e., not using infinity structures) for building a different homotopy type $X$ with the rational homotopy Lie algebra of $P$ would very likely involve attaching infinitely many cells to abelianize the Lie algebra $\pi_*(\Omega X)\otimes \Q$.
\end{remark}

Baues' theorem can be proven as a straightforward consequence of a degree argument for $A_\infty$ structures (see for instance \cite[Thm. 11]{Kad09}). The Eckmann-Hilton dual to Baues' theorem is given below. Its proof is a straightforward consequence of the fact that the whole $L_\infty$ structure vanishes by a degree argument. Observe that Quillen's theory does not require finite type hypothesis.

\begin{theorem} \label{IntriCo2}An arbitrary product of simply connected even dimensional Eilenberg-Mac Lane spaces $\prod_{i\in I} K(\Q,n_i)$ is intrinsically coformal.
\end{theorem}

\section{Coformality and the collapse of the Quillen spectral sequence}\label{ssQuillen}

The Quillen spectral sequence of a simply connected space $X$, introduced in \cite{Qui69}, has been studied in \cite{All73,All77,Tan83,Ret85}. Recently (\cite{Ber17}), a version of this spectral sequence capturing group actions on $X$ which do not necessarily preserve the base point has been derived. It is claimed in \cite[\S 2.3]{Ber17} that the collapse of the Quillen spectral sequence is weaker than the coformality of $X$, a true fact proven in this section. There is literature relating the collapse of a certain spectral sequence with (co)formality. For instance, a characterization of the formality of a DGL (in fact, of an $L_\infty$ algebra) has been given in terms of the Chevalley-Eilenberg spectral sequence in \cite{Man15}. Similar claims, extending these results for algebras over operads, have been recently given in \cite{Mar17}. \\

In this section, extending in the obvious way the Quillen spectral sequence to arbitrary $L_\infty$ algebras, we show that even restricting to DGL's, formality of one such a DGL is not equivalent to the collapse at the second page of its Quillen spectral sequence. Example \ref{Ejemplo} complements the references mentioned above and merges the higher Whitehead brackets into the picture.\\

For a coaugmented graded coalgebra $C$ with comultiplication $\Delta$ and coaugmentation kernel $\bar C$, the \emph{reduced diagonal} $\bar \Delta:\bar C \to \bar C^{\otimes 2}$ is given by $\bar \Delta (x)=\Delta (x)-\left(1\otimes x + x\otimes 1\right)$. For $n\geq0$, the \emph{iterated reduced diagonals} $\Delta^n:\bar C\to \bar C^{\otimes n+1}$ are $\Delta^0=\id, \Delta^1=\bar \Delta$ and recursively for every $n \geq 2$ by $\Delta^n = \left(\bar \Delta\otimes \id^{\otimes n-1}\right)\Delta^{n-1}$. There is an ascending differential graded coalgebra filtration $\mathcal F$ given by $F_pC=\Ker\left(\Delta^p\right)$, $p\geq1$. $C$ is \emph{conilpotent} if this filtration exhausts $\bar C.$ Every graded coalgebra of the form $C=\K\oplus C_{>0}$ is conilpotent. In particular, for a reduced $L_\infty$ algebra $L$, its Quillen chains $\mathcal C(L)$ is conilpotent. In case $C=\Lambda V$ for some graded vector space $V$, it follows that $F_pC=\Lambda^{\leq p}V$. The \emph{Quillen spectral sequence} of the $L_\infty$ algebra $L$ is the coalgebra spectral sequence determined by the differential graded filtered module $(\mathcal C (L), \mathcal F).$ The first two pages of this spectral sequence are $(E^0,d^0)\cong (\Lambda sL,\delta_1)$ and $(E^1,d^1)\cong (\Lambda sH,H(\delta_2))$. For a simply connected space $X$ with Lie model $L$, the Quillen spectral sequence converges to $H_*(X;\Q)$ as graded coalgebras. 

An $L_\infty$ algebra $L$ is \emph{formal} if it is weakly equivalent to its homology endowed with an $L_\infty$ structure for which all higher brackets vanish except for possibly the bracket $\ell_2$ induced by $L$.

\begin{proposition}\label{coformalImplica} If $L=L_{\geq 0}$ is a formal $L_\infty$ algebra, then its Quillen spectral sequence collapses at $E^2$.
\end{proposition}
\begin{proof} Let $L$ be a formal $L_\infty$ algebra. Then, there exist $L_\infty$ quasi-isomorphisms 
\begin{center}
\begin{tikzcd}
L \ar[shift left]{r}{} & L' \ar[shift left]{l}{},
\end{tikzcd}
\end{center} where $L'$ is a minimal $L_\infty$ algebra with $\ell_n=0$ for every $n\geq 3.$ Therefore, there are CDGC quasi-isomorphisms 
\begin{center}
\begin{tikzcd}
\left(\Lambda sL,\delta \right) \ar[shift left]{r}{} & \left(\Lambda sL',\delta \right) \ar[shift left]{l}{}.
\end{tikzcd}
\end{center} By comparison, both spectral sequences are isomorphic from the first page. Hence, it suffices to consider the case in which $L$ has vanishing higher brackets except for possibly $\ell_2$. Recall that, if $$Z^k_p=F_p \mathcal C(L) \cap \delta^{-1}\left(F_{p-k} \mathcal C(L)\right)\quad \textrm{and} \quad D_p^k=F_p\mathcal C(L) \cap \delta\left(F_{p+k}\mathcal C(L)\right),$$ then the differential $d^k$ in the Quillen spectral sequence is induced by the restrictions $ \delta:Z^k_p \to Z^k_{p-k}$ as in the diagram below:
\begin{center}
\begin{tikzcd}[row sep=huge, column sep=huge]
Z^k_p \ar{r}{\delta} 	\ar[two heads]{d}		&		Z^2_{p-2} \ar[two heads]{d}\\
E^k_p = Z^k_p / Z^{k-1}_{p-1}+D^{k-1}_{p} \ar[dashed]{r}{d^k} &   E^k_{p-k}= Z^k_{p-k} / Z^{k-1}_{p-k-1}+D^{k-1}_{p-k}
\end{tikzcd}
\end{center}  Let $k\geq 2$. To see that $d^k=0$, given $x\in Z^k_p=\Lambda^{\leq p} sL \cap \delta^{-1}\left( \Lambda^{\leq p-k} sL\right)$, we find a representative $y$ of the class $\overline x\in E^k_p$ for which $\delta(y)\in Z^{k-1}_{p-k-1}+D^{k-1}_{p-k}=\Lambda^{\leq p-k-1} sL \cap \delta^{-1}\left( \Lambda^{\leq p-k-1}sL \right)+ \Lambda^{\leq p-k} sL \cap \delta \left( \Lambda^{\leq p-1} sL \right).$ Indeed, write $x=x_1 + \cdots + x_p$ with each $x_i\in  \Lambda^i sL.$ Since $\delta(x)\in \Lambda^{\leq p-k}sL,$ necessarily $\delta\left(x_{p-k+1}+\cdots + x_p\right)=0.$ Hence, $y=x-\left(x_{p-k+1}+\cdots + x_p\right)$ is one such representative.
\end{proof}

The result above is a particular case of the following more general statement, whose proof is analogous: for an $L_\infty$ algebra $L$ all of whose brackets vanish up to $\ell_k$, its Quillen spectral sequence has pages $E^0 = \cdots = E^{k-1} \cong \Lambda sL$, and $d^{k-1}$ is identified with $\delta_k = s\ell_k (s^{-1})^{\otimes k}.$ An important consequence of this rather elementary fact is that the least $n$ for which $\ell_n$ is non trivial is an invariant of the $L_\infty$ quasi-isomorphism class of minimal $L_\infty$ algebras.\\

We give the promised example.

\begin{example}\label{Ejemplo} Let $X$ be the rationalization of the total space in a fibration of the sort $$S^7 \to E \to K(\Z,2)\times K(\Z,4).$$ Then, the Quillen spectral sequence of $X$ collapses at $E^2$ and all higher Whitehead products vanish, yet $X$ is not coformal.
\end{example}

\begin{proof}  Basic facts on rational homotopy theory imply that the rational homotopy Lie algebra $\pi_*(\Omega X)\otimes \Q$ is three dimensional on $x,y,z$ of degrees $1,3,6$ respectively, and carries an $L_\infty$ structure for which all higher brackets vanish except for $$\ell_2(y,y)=\ell_3(y,x,x)=z.$$ Equivalently, its Sullivan minimal model is the free CDGA $\left(\Lambda V,d\right)$ generated by $x,y,z$ of degrees $2,4,7$, respectively, with differential given by $dx=dy=0$ and $dz=y^2+yx^2.$  \\ We start by proving that $X$ is not coformal by showing that it does not admit a purely quadratic differential (Corollary \ref{SulCuad}). The commutative graded algebra automorphisms of $\Lambda V$ are of the form
\begin{equation*}
f(x)=ax, \quad f(y)=by+cx^2, \quad \textrm{and} \quad f(z)=ez
\end{equation*} for any $a,b,c,e\in \Q$ with $a,b,e\neq0$. The inverse of such an automorphism is  
\begin{equation*}
f^{-1}(x)=\frac{1}{a}x, \quad f^{-1}(y)=\frac{1}{b}y-\frac{c}{a^2b}x^2, \quad \textrm{and} \quad f^{-1}(z)=\frac{1}{e}z.
\end{equation*} The differentials $d'$ on $\Lambda V$ for which $f:\left(\Lambda V,d'\right) \to \left(\Lambda V,d\right)$ is a CDGA isomorphism are of the form $d'=fdf^{-1}.$ By uniqueness of the minimal Sullivan model, to prove our claim it is enough to compute all these possible differentials $d'.$ Since $x$ and $y$ are cycles, any differential on $\Lambda V$ is determined by its value on $z$:
\begin{align*}
d'(z) &= fdf^{-1}(z) = \frac{1}{e} fd(z) = \frac{1}{e} f\left( y^2+yx^2 \right) = \frac{1}{e}\left(fy\right)^2 + \frac{1}{e}\left(fy\right)\left(fx\right)^2\\
	&= \frac{1}{e} \left(by+cx^2\right)^2 + \frac{1}{e} \left(by+cx^2\right)\left(a^2x^2\right)\\
	&= \frac{1}{e} \left(b^2y^2 + c^2x^4 + 2bcyx^2\right) + \frac{a^2}{e} \left(byx^2+cx^4\right)\\
	&= \frac{b^2}{e} y^2 + \frac{b(2c+a^2)}{e} yx^2 + \frac{c(c+a^2)}{e}x^4.
\end{align*} For $d'$ to be purely quadratic, it should be that $b\left(2c+a^2\right)=c\left(c+a^2\right)=0$. Since $b\neq 0$, it follows that $a^2=-2c$. And, since $a\neq 0$, also $c\neq 0$. But $c=2c$, a contradiction. Thus $X$ is not coformal.

We prove next that the Quillen spectral sequence of $X$ collapses at $E^2$. The coalgebra model $\Lambda U$ of $X$ is the linear span of $$\left\{ \ x^ny^m,\ x^ny^mz  \ \mid \ n,m \geq 0\ \right\},$$ with codifferential $\delta$ determined by 
\begin{equation*}
\delta(x^ny^m) = \displaystyle  {{\left(\frac{m(m-1)}{2}\right)} x^n y^{m-2}z + {\left(\frac{nm(m-1)}{2}\right)} x^{n-1}y^{m-2}z \quad \textrm{ if } m\geq 2 }, 
\end{equation*} and zero otherwise. By definition, proving that $d^2=0$ is the same as proving that, for any given $p\geq 0$,
\begin{equation}\label{SeTrivializa}
\Phi \in Z^2_p \quad \Rightarrow \quad \delta(\Phi) \in Z^{1}_{p-3} + D^{1}_{p-2}.
\end{equation} If $\Phi$ is a cycle, we are done. Otherwise, assume that $\Phi\in \Lambda^{\leq p} U$ is of the form $\Phi=x^ny^m$ with $n+m \leq p.$ In this case, $$\delta(\Phi) = {\left(\frac{m(m-1)}{2}\right)} x^n y^{m-2}z + {\left(\frac{nm(m-1)}{2}\right)} x^{n-1}y^{m-2}z.$$ If $n+m=p$, then $\Phi\notin Z^2_p$. Thus, it suffices to check equation (\ref{SeTrivializa}) for $n+m\leq p-1$. But then, by the very definition, $\Phi\in \Lambda^{\leq p-1} U$ is such that $\delta(\Phi)\in \Lambda^{\leq p-2}U$, that is, $\delta(\Phi)\in D_{p-2}^1$, and implication (\ref{SeTrivializa}) is satisfied. For $r >2,$ proving that $d^r=0$ is the same as proving, for any given $p\geq 0$, that
\begin{equation*}
\Phi \in Z^r_p \quad \Rightarrow \quad \delta(\Phi) \in Z^{r-1}_{p-r-1} + D^{r-1}_{p-r}.
\end{equation*} The codifferential $\delta$ has only homogeneous components $\delta_2$ and $\delta_3$. All elements in $Z^r_p$ are therefore cycles, and the condition is then trivially satisfied. To conclude that all higher Whitehead products vanish, observe that by a classical theorem of Allday (\cite[Thm. 4.1]{All77}), the differential $d^{k-1}$ of the spectral sequence represents higher Whitehead products of order $k$ whenever these are defined. Since the differentials vanish for $k\geq 2$, the claim follows. 
\end{proof}

\section{Higher Whitehead products and Sullivan $L_\infty$ algebras}\label{SullivanYLInf}

The detection of higher Whitehead products in Sullivan models was studied by P. Andrews and M. Arkowitz in \cite{Ark78}. In that work, the authors very explicitly determined how to read off the (ordinary, as well as) higher Whitehead products from the differential of the minimal Sullivan model of a simply connected finite type rational complex. In this Section, we recover and slightly generalize this result by relying on Quillen's formulation of rational homotopy theory and the results of \cite{HigherWhitehead}.

To achieve the mentioned result, we recall in Section \ref{Arkowitz} the necessary background. Then we recall in Section \ref{SullivanLInfinity} that, under certain assumptions which we fix for the rest of the article, $L_\infty$ algebras uniquely correspond to Sullivan algebras (\cite{Bui13,Ber15}). Finally, we achieve the main goal of this section by showing how some of our previous results in \cite{HigherWhitehead} generalize and/or complement the main result of \cite{Ark78}. All spaces in this section are assumed to be simply connected rational CW-complexes of finite type.

\subsection{Higher Whitehead products in Sullivan models} \label{Arkowitz} Let $n_1, ... , n_r$ be fixed positive integers, and denote by $M_r(\Q)$ the square matrices of size $r$ with rational coefficients. Define the map 
$$\widetilde\rho\colon M_r(\Q)\longrightarrow \Q,\qquad \widetilde\rho(A)=\sum_{\sigma\in S_r} \varepsilon_            \sigma \, a_{1\sigma(1)} \cdots a_{r\sigma(r)},$$ where $A=(a_{ij})$, and $\varepsilon_\sigma$ is the sign arising from associating to each element  $a_{ij}$ a generator $w_j$ of degree $n_j$ in the free graded commutative algebra $\Lambda \left\{w_1,...,w_r\right\}$ and writing $w_1 \cdots w_r = \varepsilon_\sigma \,w_{\sigma(1)} \cdots w_{\sigma(r)}.$ More precisely, $\varepsilon_\sigma$ is the parity of
$$\sum_{i=1}^{r-1} \sum_{\substack{1\leq j <\sigma^{-1}(i) \\ \sigma(j)>i}} n_in_{\sigma(j)},$$ and it is $1$ whenever the above sum is empty. 

\begin{remark}\rm \label{remark}In other words, if one considers $\Q$-linear combinations $$y_1=a_{11}w_1+ \cdots + a_{r1}w_r , \quad \,\ . \ . \ . \ ,  \quad y_r=a_{r1}w_1+ \cdots + a_{rr}w_r$$ in the commutative graded algebra $\Lambda \left\{w_1,..., w_r\right\}$, and one forms the matrix $A=(a_{ij})$, then $\widetilde\rho(A)$ is simply the coefficient of the term $w_1\cdots w_r$ in the product $y_1\cdots y_r$, with the correct Koszul sign. In particular, if the integers $n_1,\dots,n_r$ are all odd, i.e., $w_1,\dots,w_r $ are oddly graded, then $\widetilde\rho$ is simply the determinant. Hence, in general,  we may think of $\widetilde \rho$ as a ``graded determinant''. In particular, observe that, if a row or a column of some matrix $A$ is the zero vector then $\widetilde\rho(A)=0$. However, in general, $\widetilde\rho$ does not vanish on matrices whose rows (or columns) are linearly dependent.
\end{remark}

Let $\left(\Lambda V,d\right)$ be the minimal Sullivan model of the simply connected complex $X$, fix a KS-basis  $\left\{v_i\right\}_{i\ge 1}$ of $V$ and homotopy classes $x_j\in \pi_{n_j}(X)$, for $1 \leq j \leq r$, and define a map 
$$
\rho=\rho_{\{v_i\}\{x_1,\dots,x_r\}}\colon \Lambda^{\geq r}V \longrightarrow \Q
$$
as follows: for a given $\Phi\in \Lambda^{\geq r }V$, write $$\Phi= \sum_{i_1 \leq \cdots \leq i_r} \lambda_{i_1...i_r} \, v_{i_1}\cdots v_{i_r} + \beta,$$ where $\lambda_{i_1...i_r}\in \Q$ and $\beta \in \Lambda^{>r}V$. For each $i_1 \leq \dots \leq i_r$ let $A_{i_1...i_r}\in M_r(\Q)$ be the matrix whose entries are given by $a_{pq}=\langle v_{i_p};x_q \rangle$, where $\langle\,\,;\,\rangle$ is the Sullivan pairing (\cite[Chap. 13 (c)]{Yve12}). Then,  $$\rho\left(\Phi\right)=\sum_{ i_1 \leq \cdots \leq i_r} \lambda_{i_1...i_r}\,   \widetilde\rho\left(A_{i_1...i_r}\right).$$

\begin{remark}\label{remarko2}\rm  Whenever the homotopy classes $x_1,\dots,x_r$ are linearly independent, $\rho(\Phi)$ has the following interpretation: identify $x_1,\dots,x_r$ with vectors $w_1,\dots,w_r$ of $V$ through the Sullivan pairing, and extend this to a new basis of $V$ which in turn produces a new basis of $\Lambda V$. Write
$$
\Phi=\lambda\, w_1\cdots w_r+\Gamma
$$
as a linear combination of this basis. Then, it follows from Remark \ref{remark} that $ \rho(\Phi)$ is precisely $\lambda$. 
\end{remark}

We are ready to recall  how the differential of a Sullivan minimal model captures higher Whitehead products. Denote $N=n_1+\cdots + n_r.$

\begin{theorem}{\em \cite[Thm. 5.4]{Ark78}} \label{TeoArk} Let $\left(\Lambda V,d\right)$ be the Sullivan minimal model of the simply connected complex $X$. Fix  a KS-basis $\{v_i\}$ of $V$ and homotopy classes $x_j\in \pi_{n_j}(X)$,  $1\leq j \leq r$, such that $\left[x_1,\dots,x_r\right]$ is defined. Let $\Phi\in(\Lambda V)^{N-1}$ such that $d\Phi\in\Lambda^{\ge r} V$. Then, for every $x\in \left[x_1,\dots,x_r\right]$, 
$$
\langle \overline \Phi;x\rangle = (-1)^\alpha  \rho(d\Phi),$$ where $\rho=\rho_{\{v_i\}\{x_1,\dots,x_r\}}$, $\alpha=\sum_{i<j}n_in_j$ and $\overline\Phi\in V$ is the linear part of $\Phi$. In particular, if $v\in V^{N-1}$ is such that $dv\in \Lambda^{\geq r}V$, then
$$
\langle  v;x\rangle = (-1)^\alpha  \rho(dv).$$ 
\end{theorem}
The rational number $ \rho(dv)$ depends on the chosen basis for $V$ but only of the $r$th part $d_r$ of the differential $d$.  As an illustrative example, we keep the notation of this theorem in the following result, which is obvious in view of  Remark \ref{remarko2}:

\begin{corollary} Let $v\in V$ be such that, in the chosen homogeneous basis of $V$, 
$$dv= \sum_{i_1 \leq \cdots \leq i_r} \lambda_{i_1...i_r} \, v_{i_1}\cdots v_{i_r} + \beta,\qquad\beta\in \Lambda^{>r} V.
$$Fix $i_1\le \dots\le i_r$ and let  $x_{i_1},\dots,x_{i_r}$ be the homotopy classes  dual to $v_{i_1},\dots,v_{i_r}$ through the Sullivan pairing. Then, for each $x\in[x_{i_1},\dots,x_{i_r}]$,
$$
\langle  v;x\rangle = (-1)^\alpha \lambda_{i_1...i_r}.\eqno{\square}
$$
\end{corollary}

\subsection{Sullivan $L_\infty$ algebras} \label{SullivanLInfinity}

To compare the results of the previous section with the $L_\infty$ structures on $\pi_*(\Omega X)\otimes \Q$, we need to recall when and how Sullivan algebras correspond to $L_\infty$ structures. We refer the reader to \cite{Bui13} and \cite{Ber15} in order to find the most general results and a meticulous study of the subtleties concerning this duality.

Recall that, by Theorem \ref{linfinito}, an $L_\infty$ structure on  a graded vector space $L$ uniquely corresponds to a codifferential $\delta$ on the coalgebra  $\Lambda sL$. If, in addition, $L=L_{\ge 0}$ is finite type and non-negatively graded, then $sL$ is concentrated in positive degrees and therefore $\Lambda sL$ is also finite type and connected. On the other hand, it is well known that dualization provides a one to one correspondence between finite type CDGA's and CDGC's. Hence, an $L_\infty$ structure on a finite type, non negatively graded vector space $L$ corresponds to the CDGA  $\bigl((\Lambda sL)^\sharp,\delta^\sharp\bigr)$ dual to the CDGC $(\Lambda sL,\delta)$. Under the same finiteness and bounding hypothesis, a slight generalization of the proof of \cite[Lemma 23.1]{Yve12} provides an isomorphism of commutative graded algebras
$$
\Lambda (sL)^\sharp\cong (\Lambda sL)^\sharp
$$
via the following pairing \cite[p. 294]{Yve12}, which is essential in what follows: for each $k\ge 1$, denote 
$V=(sL)^\sharp$ and define
\begin{equation}\label{pairingoriginal}
\langle\,\,;\,\rangle\colon \Lambda^kV\times \Lambda^ksL\longrightarrow\Q,
\end{equation}
 $$\langle v_1\cdots v_k ; sx_k \wedge ... \wedge sx_1\rangle=\sum_{\sigma\in S_k}\varepsilon_\sigma \langle v_{\sigma(1)};sx_1\rangle \cdots \langle v_{\sigma(k)};sx_k\rangle,$$
 where $v_1\dots v_k=\varepsilon_\sigma\, v_{\sigma(1)}\dots  v_{\sigma(k)}$. Through this isomorphism, the differential $\delta^\sharp$ in $(\Lambda sL)^\sharp$ becomes the differential  $d$ in $\Lambda(sL)^\sharp=\Lambda V$ whose $k$th part $d_k$ satisfies
\begin{equation}\label{pairing}
\langle d_k v;sx_1\wedge ... \wedge sx_k \rangle
=\varepsilon \langle v;s\ell_k( x_1,\dots, x_k) \rangle
\end{equation}
where  $v\in (sL)^\sharp$, $x_i\in L$, $i=1,\dots,k$ and $\varepsilon=(-1)^{|v|+\sum_{j=1}^{k-1}(k-j)|x_j|}$. Via this equality, for each $v\in (sL)^\sharp$ and any $k\ge 1$, $d_k v$ is well defined and vanishes for $k$ big enough, due to the finite type and the bounding assumption respectively. Summarizing, 

\begin{proposition}$L_\infty$ algebras on a finite type, non negatively graded vector space $L$ are in one to one correspondence with differentials $d$ on the free commutative graded algebra $\Lambda V$ generated by $V=(sL)^\sharp$. Explicitly, for each $k\ge 1$, the $k$th bracket $\ell_k$ on $L$ and the $k$th part $d_k$ determine each other via the formula (\ref{pairing}) above. \hfill$\square$
\end{proposition}

From now on, $L_\infty$ algebras are all non negatively graded and of finite type. We can easily detect when a given $L_\infty$ algebra on $L$ provides a Sullivan algebra through the correspondence above. 

\begin{proposition} Let $L$ be an $L_\infty$ algebra and let $(\Lambda V,d)$ be the corresponding CDGA. Then, the following are equivalent: 
\begin{itemize}
	\item[(1)] $(\Lambda V,d)$ is a Sullivan algebra.
	\item[(2)] There exists an ordered homogeneous basis $\left\{x_i\right\}_{i\in I}$ of $L$ such that for every $k\geq 1$, the class of $\ell_k\left(x_{i_1},...,x_{i_k}\right)$ vanishes in the quotient $L/L^{>j}$, where $j=\max \left\{i_1,...,i_k\right\}$ and $L^{>j}$ is the linear span of $\left\{x_i \mid i>j\right\}$.
	\item[(3)] $L$ is ``degree-wise nilpotent'': in the \emph{lower central series} $L=\Gamma^0 L \supseteq \Gamma^1 L \supseteq ... $ of $L$, where each $\Gamma^k L$ is the subspace of $L$ generated by all possible brackets using at least $k$ elements of $L$, for every $n$ there exists some $k$ with the property that $(\Gamma^k L)_n=0.$
\end{itemize}
\end{proposition}
\begin{proof} The equivalence between (1) and (2) arises simply from translating the property defining a Sullivan algebra to the brackets of $L$ through the formula (\ref{pairing}). On the other hand, the equivalence between (2) and (3) is precisely \cite[Thm. 3.2]{Ber15}.
\end{proof}

A \emph{Sullivan $L_\infty$ algebra} is an $L_\infty$ algebra satisfying any of the above equivalent conditions.\\

Let $(\Lambda V,d)$ be the minimal Sullivan algebra  equivalent to the minimal Sullivan $L_\infty$ algebra $\left(L,\left\{\ell_k\right\}\right)$ and fix a KS-basis  $\{v_i\}$ of $V$. Fixing elements $x_1,\dots,x_r\in L$ (not necessarily  of the given basis), the map 
$$
 \rho=\rho_{\{v_i\}\{x_1,\dots,x_r\}}\colon \Lambda^{\geq r}V \to \Q
$$
of the past section can be defined using the same procedure: write any element $\Phi\in \Lambda^{\geq r }V$ as $$\Phi= \sum_{i_1 \leq \cdots \leq i_r} \lambda_{i_1...i_r} \, v_{i_1}\cdots v_{i_r} + \beta,$$ where $\lambda_{i_1...i_r}\in \Q$ and $\beta \in \Lambda^{>r}V$, let $A_{i_1...i_r}\in M_r(\Q)$ be the matrix whose entries are given by $a_{pq}=\langle v_{i_p};sx_q \rangle$, and recall that $V=(sL)^\sharp$. Then,  $$ \rho\left(\Phi\right)=\sum_{ i_1 \leq \cdots \leq i_r} \lambda_{i_1...i_r}\,   \widetilde\rho\left(A_{i_1...i_r}\right).$$

\begin{remark}\rm \label{remarko3} In view of the pairing (\ref{pairingoriginal}), for any $v\in V$ and any $sx_{1},\dots,sx_{k}\in sL$,  a short computation shows that, for $\rho=\rho_{\{v\}\{x_1,\dots,x_k\}}$,
$$
\langle d_k v;sx_1\wedge ... \wedge sx_k \rangle= \rho(d_kv).
$$
\end{remark}

\subsection{Extending Andrews-Arkowitz's theorem to $L_\infty$ algebras}

Let $L$ be a Lie model of the simply connected finite type complex $X$, consider a homotopy retract $(L,H,i,q,K)$ of 
$L$, and let $\{\ell_k\}$ be the corresponding $L_\infty$ structure on $H$ (see \cite[p. 22]{HigherWhitehead}). Observe that  $H$ is a Sullivan $L_\infty$ algebra whose associated Sullivan algebra $(\Lambda V,d)$ is the minimal model of $X$ for which we fix a KS-basis. Then, the translation of Theorem  \ref{TeoArk} in this context reads:

\begin{theorem}\label{TeoArkInf} Let $x_j\in H$,  $1\leq j \leq r$, be such that $\left[x_1,\dots,x_r\right]$ is defined. Let $v\in V^{N-1}$ be such that $dv\in\Lambda^{\ge r}V$. Then, for every $x\in \left[x_1,\dots,x_r\right]$, 
$$
\langle  v;sx\rangle = \varepsilon \langle  v;s\ell_r(x_1,\dots,x_r)\rangle.$$ 
\end{theorem}

\begin{proof}
Indeed, recall that $V=(sH)^\sharp\cong\pi_*(X)\otimes \Q.$ Hence, Theorem \ref{TeoArk} states that $\langle  v;x\rangle=(-1)^\alpha\rho(dv)$. But, $\rho(dv)=\rho(d_kv)$, and in view of Remark \ref{remarko3} and formula (\ref{pairing}),
$$
(-1)^\alpha\rho(dv)=(-1)^\alpha\,\langle d_k v;sx_1\wedge ... \wedge sx_k \rangle=\varepsilon \langle v;s\ell_k( x_1,\dots, x_k) \rangle.
$$
\end{proof}

\begin{remarks}\label{Remarkillos} \rm  
(1) Observe that Theorem \ref{TeoArkInf}, and hence Theorem \ref{TeoArk}, can be easily deduced from our Proposition 3.1 of \cite{HigherWhitehead}. Indeed, via this result, and for any $x\in [x_1,\dots,x_r]$,
$$
\varepsilon \ell_r\left(x_1,...,x_r\right) = x + \sum_{j=2}^{r-1} \ell_j\left(\Phi_j\right), \qquad\Phi_j\in H^{\otimes j}.
$$
Hence, for any $v\in V^{N-1}$, 
$$
\left \langle v;sx \right\rangle = \varepsilon \left \langle v;s\ell_r\left(x_1,...,x_r\right)\right\rangle - \sum_{j=2}^{r-1}\left\langle v; s\ell_j\left(\Phi_j\right) \right\rangle.
$$
However, if $dv\in \Lambda^{\ge r}V$, then $d_j(v)=0$ for every $j<r$ and therefore, in view of formula (\ref{pairing}), the second term vanishes. Therefore,
$$
\langle  v;sx\rangle = \varepsilon\langle  v;s\ell_r(x_1,\dots,x_r)\rangle,$$ 
which is the statement of Theorem \ref{TeoArkInf}.

(2) Observe also that  \cite[Thm. 3.3]{HigherWhitehead} is a generalization of Theorem \ref{TeoArkInf}, and hence of Theorem \ref{TeoArk}, under the presence of an adapted homotopy retract. Indeed, given $x\in[x_1,\dots,x_r]$, \cite[Thm. 3.3]{HigherWhitehead} asserts that
$$
\ell_r(x_1,\dots,x_r)=x,
$$
whenever the $L_\infty$ structure on $H$ arises from a homotopy retract adapted to $x$. That is,
$$
\langle v;sx\rangle=\langle  v;s\ell_r(x_1,\dots,x_r)\rangle=(-1)^\alpha\rho(dv),\quad\text{for all $v\in V^{N-1}$,}
$$
and not only for those $v$ with $dv\in \Lambda^{\ge r}V$.

\end{remarks}

\bibliographystyle{plain}
\bibliography{MyBib}

\begin{thebibliography}{10}

\bibitem{All73}
C.~Allday.
\newblock Rational {W}hitehead products and a spectral sequence of {Q}uillen.
\newblock {\em Pac. J. Math.}, 46(2):313--323, 1973.

\bibitem{All77}
C.~Allday.
\newblock Rational {W}hitehead products and a spectral sequence of {Q}uillen
  {I}{I}.
\newblock {\em Houston J. Math.}, 3(3):301--308, 1977.

\bibitem{Ark78}
P.~Andrews and M.~Arkowitz.
\newblock Sullivan’s minimal models and higher order {W}hitehead products.
\newblock {\em Canad. J. Math}, 30:961--982, 1978.

\bibitem{Bau77}
H.~J. Baues.
\newblock Rationale {H}omotopietypen.
\newblock {\em Manuscripta Math}, 20(2):119--131, 1977.

\bibitem{HigherWhitehead}
F.~Belchí, U.~Buijs, J.~M. Moreno-Fernández, and A.~Murillo.
\newblock Higher order {W}hitehead products and {$L_\infty$} structures on the
  homology of a {D}{G}{L}.
\newblock {\em Linear Algebra Appl.}, 520:16 -- 31, 2017.

\bibitem{Ber14}
A.~Berglund.
\newblock Koszul spaces.
\newblock {\em Trans Am Math Soc.}, 366(9):4551--4569, 2014.

\bibitem{Ber15}
A.~Berglund.
\newblock Rational homotopy theory of mapping spaces via {L}ie theory for
  ${L}_\infty$-algebras.
\newblock {\em Homology, Homotopy Appl.}, 17(2):343--369, 2015.

\bibitem{Ber17}
A.~Berglund and K.~B{\"o}rjeson.
\newblock Free loop space homology of highly connected manifolds.
\newblock In {\em Forum Math}, volume~29, pages 201--228. De Gruyter, 2017.

\bibitem{Ber17A}
A.~Berglund and K.~B{\"o}rjeson.
\newblock Koszul {A}-infinity algebras and free loop space homology.
\newblock {\em arXiv preprint arXiv:1711.06496}, 2017.

\bibitem{Bui13}
U.~Buijs and A.~Murillo.
\newblock Algebraic models of non-connected spaces and homotopy theory of
  {${L}_\infty$} algebras.
\newblock {\em Adv. in Math.}, 236:60--91, 2013.

\bibitem{Cor86}
L.~A. Cordero, M.~Fern{\'a}ndez, and A.~Gray.
\newblock Symplectic manifolds with no {K}ahler structure.
\newblock {\em Topology}, 25(3):375--380, 1986.

\bibitem{Del75}
P.~Deligne, P.~Griffiths, J.~Morgan, and D.~Sullivan.
\newblock Real homotopy theory of {K}{\"a}hler manifolds.
\newblock {\em Invent. Math.}, 29(3):245--274, 1975.

\bibitem{Yve12}
Y.~F{\'e}lix, S.~Halperin, and J-C Thomas.
\newblock {\em Rational homotopy theory}, volume 205.
\newblock Springer Science \& Business Media, 2012.

\bibitem{Gol88}
W.~M. Goldman and J.~J. Millson.
\newblock The deformation theory of representations of fundamental groups of
  compact {K}{\"a}hler manifolds.
\newblock {\em Publ. Math.}, 67:43--96, 1988.

\bibitem{Hal79}
S.~Halperin and J.~Stasheff.
\newblock Obstructions to homotopy equivalences.
\newblock {\em Adv. Math.}, 32(3):233--279, 1979.

\bibitem{Hin97}
V.~Hinich.
\newblock Homological algebra of homotopy algebras.
\newblock {\em Comm. in algebra}, 25(10):3291--3323, 1997.

\bibitem{Kad09}
T.~Kadeishvili.
\newblock Cohomology {$C_\infty$}-algebra and rational homotopy type.
\newblock In {\em Algebraic topology---old and new}, volume~85 of {\em Banach
  {C}enter {P}ubl.}, pages 225--240. Polish {A}cad. {S}ci. {I}nst. {M}ath.,
  {W}arsaw, 2009.

\bibitem{Kon03}
M.~Kontsevich.
\newblock Deformation quantization of {P}oisson manifolds.
\newblock {\em Lett Math Phys.}, 66(3):157--216, 2003.

\bibitem{Lad95}
T.~Lada and M.~Markl.
\newblock Strongly homotopy {L}ie algebras.
\newblock {\em Commun Algebra}, 23(6):2147--2161, 1995.

\bibitem{Lod12}
J-L. Loday and B.~Vallette.
\newblock {\em Algebraic operads}, volume 346.
\newblock Springer, Heidelberg, 2012.

\bibitem{Man15}
M.~Manetti.
\newblock On some formality criteria for {D}{G}-{L}ie algebras.
\newblock {\em J. Algebra}, 438:90--118, 2015.

\bibitem{Mas58}
W.~S. Massey.
\newblock Some higher order cohomology operations.
\newblock In {\em Int. Symp. Alg. Top. Mexico}, pages 145--154. Citeseer, 1958.

\bibitem{Mar17}
V.~Melani and M.~Rubi{\'o}.
\newblock Formality criteria for algebras over operads.
\newblock {\em arXiv preprint arXiv:1712.09229}, 2017.

\bibitem{Nei78}
J.~Neisendorfer and T.~Miller.
\newblock Formal and coformal spaces.
\newblock {\em Illinois J. Math.}, 22(4):565--580, 1978.

\bibitem{Ouk78}
A.~Oukili.
\newblock Sur l'homologie d'une alg\`ebre diff\'erentielle de {L}ie.
  {U}niversit\'e de {N}ice, {P}h{D} {T}hesis, 1978.

\bibitem{Pap04}
S.~Papadima and A.~I. Suciu.
\newblock Homotopy {L}ie algebras, lower central series and the {K}oszul
  property.
\newblock {\em Geom. Topol.}, 8:1079--1125, 2004.

\bibitem{Qui69}
D.~Quillen.
\newblock Rational homotopy theory.
\newblock {\em Ann. of Math.}, pages 205--295, 1969.

\bibitem{Ret85}
V.~S. Retakh.
\newblock Massey operations in {L}ie superalgebras and differentials of the
  {Q}uillen spectral sequence.
\newblock {\em Colloq. Math.}, 50(1):81--94, 1985.

\bibitem{Tan83}
D.~Tanr{\'e}.
\newblock {\em Homotopie rationnelle: modeles de {C}hen, {Q}uillen,
  {S}ullivan.}, volume vol. 1025, Springer.
\newblock Springer, 1983.

\end{thebibliography}

\bigskip

\noindent\sc{Urtzi Buijs}\\ 
\noindent\sc{Departamento de \'Algebra, Geometr{\'\i}a y Topolog{\'\i}a,\\
Universidad de M\'alaga, Ap. 59, 29080 M\'alaga, Spain}\\
\noindent\tt{ubuijs@uma.es}

\bigskip

\noindent\sc{José M. Moreno-Fernández}\\ 
\noindent\sc{Max Planck Institute for Mathematics, \\ Vivatsgasse 7, 53111 Bonn, Germany}\\
\noindent\tt{josemoreno@mpim-bonn.mpg.de}

\end{document}